\documentclass[11pt,a4paper]{amsart}

\usepackage[margin=1in]{geometry}
\usepackage[%pdftex,
bookmarks=true]{hyperref}
\usepackage[utf8]{inputenc}
\usepackage[T1]{fontenc}
\usepackage[english]{babel}
\usepackage{csquotes}
\usepackage{dsfont}
\usepackage[final]{microtype}
\usepackage{lmodern}
\usepackage{caption}
\usepackage{amsthm}
\usepackage{amsmath}
\usepackage{enumerate}
\usepackage{amssymb}
\usepackage{mathrsfs}
\usepackage{halloweenmath}
\usepackage{enumerate}
\usepackage{tikz-cd} % commutative diagrams
\usepackage{xypic} % also commutative diagrams
\usepackage{stmaryrd}
\numberwithin{equation}{section}
\usepackage{silence}
\makeindex

\newtheorem{theorem}{Theorem}[section]

\newtheorem{proposition}[theorem]{Proposition}
\newtheorem{lemma}[theorem]{Lemma}
\newtheorem{corollary}[theorem]{Corollary}
\newtheorem{claim}[theorem]{Claim}
\theoremstyle{definition}

\newtheorem{remark}[theorem]{Remark}
\newtheorem{definition}[theorem]{Definition}

\newtheorem{note}[theorem]{Note}

\newtheorem{question}[theorem]{Question}
\theoremstyle{remark}
\newtheorem*{prooflemma}{Proof of Lemma}
\newtheorem{notation}[theorem]{Notation}

\newcommand{\Z}{\mathbb{Z}}

\newcommand{\F}{\mathbb{F}}
\newcommand{\PP}{\mathbb{P}}

\newcommand{\Cs}{\mathscr{C}}
\newcommand{\Fs}{\mathscr{F}}
\newcommand{\mcC}{\mathcal{C}}

\newcommand{\cF}{\mathcal{F}}

\newcommand{\ds}{\displaystyle}
\newcommand{\barsQ}{\mid \! Q \! \mid}

\newcommand{\leftfloor}{\left\lfloor}
\newcommand{\rightfloor}{\right\rfloor}

\DeclareMathOperator{\Hom}{Hom}
\DeclareMathOperator{\rank}{rank}

\DeclareMathOperator{\Gal}{Gal}
\DeclareMathOperator{\Aut}{Aut}

\DeclareMathOperator{\PGL}{PGL}
\DeclareMathOperator{\Frob}{Frob}

\DeclareMathOperator{\Ker}{Ker}
\DeclareMathOperator{\Div}{div}
\DeclareMathOperator{\Jac}{Jac}
\DeclareMathOperator{\Spec}{Spec}

\title{Proportion of ordinarity in some families of curves over finite fields}
\author{Soumya Sankar}
\date{April 2019 }

\begin{document}
	\maketitle
	\begin{abstract}
	    A curve over a field of characteristic $p$ is called ordinary if the $p$-torsion of its Jacobian as large as possible, that is, an $\F_p$ vector space of dimension equal to its genus. In this paper we consider the following question: fix a finite field $\F_q$ and a family $\mathscr{F}$ of curves over $\F_q$. Then, what is the probability that a curve in this family is ordinary? We answer this question when $\Fs$ is either the Artin-Schreier family in any characteristic or the superelliptic family in characteristic 2. 
	\end{abstract}

	\section{Introduction}
	%    \item I think I need to understand how $\mathscr{AS}_g$ sits inside $\mathscr{M}_g$. The crucial thing to understand is when two Artin-Schreier curves are isomorphic as elements of $\mathscr{M}_g$.

	Let $\mcC$ be a smooth curve of genus $g$ over a field $k$. Then its Jacobian $\Jac(\mcC)$ is an abelian variety of dimension $g$. For each $n \in \Z_{>0}$, the $n$-torsion group scheme $\Jac[n]$ is a finite flat group scheme. When $(n,p)=1$, this group scheme is \'{e}tale, and as an abelian group, is isomorphic to $(\Z/n\Z)^{2g}$ over $\bar{k}$. When $n$ is not invertible in $k$, this group scheme is never \'{e}tale and  its isomorphism class over $\bar{k}$ depends significantly on the curve. In this paper we study the variation of this group scheme in some families of curves. In order to do so, we recall the definitions the following invariants:\\
	
	Let $G$ be a finite flat group scheme killed by $p$ over a field $k$ of characteristic $p$.
	\begin{definition}
	    We define the $a$-number of $G$ as
	    $$
	    a(G) := \dim_k \Hom(\alpha_p, G),
	    $$
	    where $\alpha_p$ is the affine group scheme $\Spec(k[x]/x^p)$, and the $\Hom$ is in the category of $k$-group schemes.
	\end{definition}

	\begin{definition}
	    The $p$-rank of $G$ is defined as $r(G)$ where
	    \[
	    G(\bar{k}) \cong (\Z/p\Z)^{r(G)}
	    %r(G) = \dim_{\Z/p} \Hom (G(\bar{k}), \Z/p\Z)
	    \]
	    as abelian groups.
	 \end{definition}
 
	 For the purpose of this paper, we will only be interested in $G = \Jac(\mcC)[p]$. In this case, it is well known that $0 \le r(G) \le g(\mcC)$ and $0 \le a(G) + r(G) \le g$. The Jacobian is called \emph{ordinary} if $r(G) = g$ or equivalently, when $a(G) = 0$ \cite{AchterPries15}. By abuse of notation, we will denote the $a(\mcC)$ and $r(\mcC)$ to be the corresponding invariants of $\Jac(\mcC)[p]$. \\
	 
	Let $\mathscr{M}_g$ denote the moduli space of smooth curves of genus $g$. The study of $\mathscr{M}_g$ can take two different directions. One is to understand $\mathscr{M}_g (\overline{k})$, i.e. to gain a geometric understanding of $\mathscr{M}_g$ for a fixed $g$. A lot of work has been done in this area, some of which we will list later in this section. The second is to understand $\mathscr{M}_g(k)$ for a given $k$. For instance, one might ask if for a fixed $q$, the limit
	\begin{equation}
	\label{EqnMaincount}
	    \lim_{X \rightarrow \infty} \frac{ \#\{\mcC \in \mathscr{M}_g(\F_q) \mid \mcC \text{ ordinary }, q^g < X \} }{ \#\{ \mcC \in \mathscr{M}_g(\F_q) \mid q^g < X \}}
	\end{equation}
	exists. And if so, what is it? Very little is known about this question.
	%It is crucial to this paper that we fix $q$ and let $g$ vary.
	In this paper, we ask what happens to the limit \eqref{EqnMaincount} when $\mathscr{M}_g$ is replaced by some other families of curves. Fix a family $\mathscr{F}$ of curves over $\F_q$ of arbitrary genus. Note that by a family, we mean a set of curves satisfying a particular property, which is not necessarily a family in any geometric sense. A typical example of a family is $\cup_{g \ge 0} \mathscr{M}_g (\F_q)$. Let $\mathscr{F}_{a} = \{\mcC \in \mathscr{F} \mid a(\mcC) = a \}$. We wish to study the probability that a randomly chosen $\mcC \in \Fs$ lies in $\Fs_0$ or in other words, what proportion of curves in the family $\mathscr{F}$ is ordinary. More precisely, we define:
	 
	 \begin{itemize}
	     \item $N(\Fs, X) = \# \{ \mcC \in \Fs \mid q^{g} < X \}$
	     \item $N(\Fs, a, X) = \# \{ \mcC \in \Fs_a \mid q^{g} < X \}$. \\
	 \end{itemize}
	 
	 \begin{question}
	 \label{mainquestion}
	    For a family $\Fs$ of curves over a fixed finite field $\F_q$, does the limit,
	    $$
	    P(\Fs,0):= \lim_{X \rightarrow \infty} \frac{N(\Fs, 0, X)}{N(\Fs, X)}
	    $$
	    exist, and if so, what is its value?
	 \end{question}

	 In Section 2, we describe two different families for which we will answer Question \ref{mainquestion}, namely:
	 	\begin{itemize}
	 	    \item Artin Schreier curves in arbitrary positive characteristic.
	 		%\item Hyperelliptic curves over a finite field of characteristic 2.
	 		\item Superelliptic curves over a finite field of characteristic 2.
	 	\end{itemize}
	 In each of the above cases, the criteria for ordinarity can be described combinatorially in terms of the ramification invariants of the curves in question.\\
	 
	 In Section 3, we prove the following results:
	 
	 \begin{theorem}[Corollary \ref{ArtinSchreierMainTheorem}]
	 \label{ASStatement1}
	    The probability that an Artin-Schreier curve (under the assumptions of \S 2) over $\F_q$, $q$ a power of $p$, is ordinary is non-zero for $p=2$ and zero for all odd primes.
	  \end{theorem}
	 
	 For $p=2$ we calculate the probability explicitly and give some values for various $q$ in \S 4. For the family of superelliptic curves, we prove:
	 
	 \begin{theorem}[Theorem \ref{Superellmaintheorem}]
	 \label{SEStatement1}
	    The probability that a superelliptic curve of prime degree over a finite field of characteristic 2 is ordinary is zero.
	 \end{theorem}

	%Say something about distribution of \F_q rational points on p-rank strata?
	
	This paper is inspired by a paper of Cais, Ellenberg and Zureick-Brown \cite{CEZB13}, which gives a heuristic for the behaviour of the $p$-divisible group of an abelian variety by proving the distribution for a random (principally quasi-polarized) Dieudonn\'{e} module. They show that the probability that such a module is ordinary  (here the $p$-rank and $a$-number of a Dieudonn\'{e} module $D$ are defined as those of $D/pD$) is
	\begin{equation}
	\label{CEZB}
	    \prod_{i=1}^{\infty} (1+ q^{-i})^{-1}.
	\end{equation}

	Further, they ask if the Dieudonn\'{e} module associated to the Jacobian of a curve behaves like a randomly chosen one, i.e. whether the limit in \eqref{EqnMaincount} equals \eqref{CEZB}. They find, via numerical experiments, that hyperelliptic curves in odd characteristic do not appear to obey their heuristics, while plane curves do. The families considered in this paper are the first known cases whose behavior provably diverges significantly from the heuristics of \cite{CEZB13}.\\
	
	Note that Theorems \ref{ASStatement1} and \ref{SEStatement1}	show that the Artin-Schreier and superelliptic families do not obey the heuristic \eqref{CEZB} and therefore do not behave randomly in the sense of \cite{CEZB13}. We explain this in greater detail in section 3.\\

	We emphasize that theorems \ref{ASStatement1} and \ref{SEStatement1} are both `large $g$-limit' results, in that they study the behavior of curves as $g \rightarrow \infty$, with $q$ fixed. The `large $q$-limit' behavior, i.e. the geometry of families of curves of a fixed genus, is usually incomparable to the large $g$-limit behavior. However, studying the former can provide some insight into it latter. To illustrate our point, we list some results here that show how different the geometry of the Artin-Schreier locus is from that of some other families of curves. It is known that the locus of ordinary curves is a non-empty Zariski open subset of $\mathscr{M}_g$ \cite{Nakajima83}. Thus for a fixed genus $g$, `most' curves of genus $g$ tend to be ordinary. Let $\mathscr{V}_{g,r}$ denote the sublocus of $\bar{\mathscr{M}}_g$ of curves of $p$-rank at most $r$. In \cite{FaberGeer04} Faber and van der Geer prove that $\mathscr{V}_{g,r}$ has codimension $g-r$. A result of Glass and Pries \cite{GlassPries05} states that $\mathscr{V}_{g,r}$ intersects the hyperelliptic locus, $\mathscr{H}_g$, inside $\bar{\mathscr{M}}_g$ in a set of dimension $g-1+r$. Since $\mathscr{H}_g$ has dimension $2g-1$, this implies that the ordinary locus is dense in $\mathscr{H}_g$. We compare this to results about $\mathscr{AS}_g$, the Artin-Schreier locus inside $\mathscr{M}_g$. In \cite{PriesZhu12}, Pries and June Zhu prove that for $p \ge 3$, the codimension of $\mathscr{V}_{g,r} \cap \mathscr{AS}_g$ inside $\mathscr{AS}_g$ is less than $g-r$. This indicates that for $p \ge 3$, the image under the Torelli morphism of $\mathscr{AS}_g$ in $\mathscr{A}_g$ is not in general position with respect to the $p$-rank stratification. Further, from results in \cite{PriesZhu12}, it follows that the ordinary locus intersects only one irreducible component of $\mathscr{AS}_g$. As $g \rightarrow \infty$, the number of components of $\mathscr{AS}_g$ increases except when $p=2$ (in which case $\mathscr{AS}_g$ is $\mathscr{H}_g$). This gives a heuristic reason for why one might expect a statement like Theorem \ref{ASStatement1}. A similar heuristic explains Theorem \ref{SEStatement1} as well, as we elaborate in remark \ref{EndremarkSE}.\\

	% In \cite{Re01}, Re gives explicit bounds on the genus of a curve with a given positive $a$-number. In particular, he proves that there are no curves with $a >0$ and $ g-a < \frac{2g}{p^2 + 2} - \frac{p-1}{p+1}$.
	 
	 %{\color{blue} My invariant product: 
	 %$$
	 %1 - q^{-1} + q^{-2} - 2q^{-3} + O(q^{-4})
	 %$$
	 
	 %CEZB constant: 
	 %$$
	 %1 - q^{-1} - q^{-3} + q^{-4} + O(q^{-5})
	 %$$
	 
	 %}

	 \subsection*{Acknowledgements} The author would like to thank Jordan Ellenberg and Rachel Pries for their invaluable help and support throughout this project, as well as Melanie Wood, Tonghai Yang and Brandon Alberts for their helpful comments. The author was partially supported by the NSF grant no. DMS-1700884.
	 
	 %put in Brandon's name later?
	 
	\section{Setup and Background}
	
	In this section, we provide the setup and background for each of the families that we will consider.

	\subsection{Artin-Schreier Curves}
	\label{ArtinSchreierBackground}
	
	Let $k$ be a perfect field of characteristic $p>0$. We now recall some facts about Artin-Schreier curves and covers. An Artin-Schreier curve $\mcC$ over $k$ is a smooth $\Z/p\Z$ cover of $\PP^1_k$. Such a curve has an affine model 
	\begin{equation}
	\label{ASmodel}
	    	y^p-y = f(x)
	\end{equation}
	where $f(x) \in k(x)$, and is equipped with a $\Z/p\Z$ action generated by $y \mapsto y+1$. An \emph{Artin-Schreier cover} is an Artin-Schreier curve along with a choice of map $\iota : \Z/p\Z \hookrightarrow \Aut(\mcC)$  and a choice of isomorphism $\mcC/ (\iota(\Z/p\Z)) \cong \PP^1$. This amounts to picking a model of the form \ref{ASmodel}. \\
	
	Let $B \subset \PP^1(\bar{k})$ be the set of poles of $f$. Then, the cover above is ramified precisely at the points in $B$ \cite{Sticht09}. For $\alpha \in B $, let
	\begin{align*}
	    x_{\alpha} = \begin{cases} \frac{1}{x-\alpha} \quad& \alpha \neq \infty\\
	    x \quad& \alpha = \infty \end{cases}
	\end{align*}
	
	Then, using a partial fraction decomposition one can write 
	\begin{equation}
	\label{partialdecomp}
	    	f(x) = \sum_{\alpha \in B} f_{\alpha}(x_{\alpha})
	\end{equation}
    
    where $f_{\alpha} \in \bar{k}[x]$. 
    
     \begin{remark}
     \label{ASRemark}
	We now make a few helpful observations about the partial fraction decomposition above.
	    \begin{enumerate}
	        \item We assume that for $\alpha \neq \infty$, $f_{\alpha}$ has no constant term.
	        \item By a transformation of the form $y \mapsto y+z$, one can assume that in $f_{\alpha}(x)$, the coefficient of $x^{ip}$ is zero for any $0 \le i \le \lfloor d_{\alpha}/p \rfloor$. In particular, we can take $d_{\alpha} \neq 0 \mod p$.
	        \item If $\alpha, \beta \in B$ are Galois conjugate, then $d_{\alpha} = d_{\beta}$. 
	        \item Let $Q$ be an irreducible polynomial in $k[x]$ whose zeroes are ramified in the Artin-Schreier cover under consideration. Then we will denote $d_{Q}$ as the degree of any $f_{\alpha}$, $\alpha$ a zero of $Q$. This is well defined by the above remark.
	    \end{enumerate}
	\end{remark}
    
    By the Riemann-Hurwitz theorem for wildly ramified covers, we know that the genus of such a curve is given by:
	
	\begin{equation}
	\label{genusArtinSchreier}
	g = \left(\frac{p-1}{2} \right) \left(-2 + \sum_{\alpha \in B} (d_{\alpha}+1)\right) = \left(\frac{p-1}{2} \right) \left(-2 + \sum_{Q \text{ irred.} \atop \text{ramified}} \deg(Q)(d_{Q}+1) + (d_{\infty}+1)\right)
	\end{equation}
	
    The following criterion for the ordinarity of an Artin-Schreier curve is well known (\cite{Crew84}, \cite{ElkinPries13}, \cite{Sub75}):

	\begin{proposition}
	    The Artin-Schreier cover $y^p-y = f(x)$ is ordinary if and only if $f$ has only simple poles.
	\end{proposition}
	
    This is equivalent to the condition that $d_{\alpha} = 1$ for each $\alpha$ in the partial fraction decomposition \eqref{partialdecomp}. 
    
   Let $\mathcal{S}$ be the set of rational functions $f(x) \in k(x)$ such that the partial fraction decomposition of $f$ satisfies the conditions (1-3) from remark \ref{ASRemark}. For simplicity, we will assume that $\infty \notin B$. This assumption is harmless, as we explain in remark \ref{RemarkArtinSchreier} and makes the computations in \S 3 much cleaner. We now restrict our attention to $k = \F_q$ and define the families for this section as follows:
   
    \begin{itemize}
        \item $\Fs = $ Set of Artin-Schreier covers $y^p-y = f(x)$, where $f(x) \in \mathcal{S}$ has no poles over $\infty \in \PP^1$. 
        \item $\Fs_0 =$ Set of all ordinary Artin-Schreier covers $y^p - y = f(x)$ with $f(x) \in \mathcal{S}$, unramified over $\infty \in \PP^1$.
    \end{itemize}

   \subsubsection{Counting curves versus counting covers}
   \label{subsubsection:AScurvesversuscovers}
   In our proof of the main theorem in \S 3, we calculate the limit from question \ref{mainquestion}  by counting polynomials in the set $\mathcal{S}$ defined above. Note that while this is not exactly the same as counting Artin-Schreier curves \emph{upto isomorphism}, it does not significantly affect the proportion of ordinarity for large enough $p$. In particular, it does not change whether the limit \ref{mainquestion} is non-zero for $p \ge 7$. We explain this below.\\
   %For $p=3,5$ such a conclusion is beyond reach right now, while for $p=2$, it is simply not true. 
   
  % By remark \ref{ASRemark}, we see that every Artin-Schreier curve over $\F_q$ is isomorphic to a curve with model:
%	$$
%	y^p - y = f(x)
%	$$
%	with $f(x) \in \mathcal{S}$.
    There is a map
	\begin{equation}
	\label{eqn:ArtinSchreierfibermap}
    \mathcal{S} \rightarrow \bigcup_{g \ge 0} \mathscr{AS}_g (\F_q)
    \end{equation}
    sending $f(x) \in \mathcal{S}$ to the curve with model $y^p-y = f(x)$. Remark \ref{ASRemark} shows that this map is surjective. We will now bound the size of the fibers. For an Artin-Schreier curve $\mcC,$ a choice of model $\mcC_f$ amounts to a choice of homomorphism $\iota : \Z/p\Z \hookrightarrow \Aut(\mcC)$ and a choice of isomorphism $\mcC/\iota(\Z/p\Z) \cong \PP^1$. For $g \ge 2$, Stichtentoth proved (see \cite{St73}, \cite{Stich73}) that $|\Aut(\mcC)| \le 16g(\mcC)^4$. We claim that the number of choices of isomorphism $\mcC/\iota(\Z/p\Z) \cong \PP^1$ is bounded uniformly. 
    
    \begin{proposition}
	\label{ASClaim}
	    Let $\mcC_f$ and $\mcC_g$ be two Artin-Schreier covers with $\phi : \mcC_f \cong \mcC_g$ such that there is a commutative diagram
	     \begin{center}
	        \begin{tikzcd}
	            \mcC_f \arrow[r, "\phi"] \arrow[d] &\mcC_g \arrow[d]\\
	            \PP^1_{\F_q} \arrow[r, "\Tilde{\phi}"] &\PP^1_{\F_q}
	        \end{tikzcd}
	    \end{center}
	    where the vertical maps are quotients by the $\Z/p\Z$ actions. 
	    Then $f(x) = u g(\gamma x)$ for some $u \in \Z/p\Z^{\times}$ and $\gamma \in \PGL_2(\F_q)$.
	\end{proposition}

%We let $\mcC_f$ denote the curve $y^p -y = f(x)$. Further, let $f(x), g(x) \in \mathcal{S}$. Then we claim that any isomorphism: $\phi: \mcC_f \cong \mcC_g$ over $\F_q$ is essentially induced by an isomorphism $\PP^1 \cong \PP^1$. This is probably well known, but since we couldn't find the explicit statement in the literature, we recall the proof here.
	
% 	\begin{claim}
% 	\label{ASClaim}
% 	    Maintaining the same notation as above, let $\mcC_f$ and $\mcC_g$ be two Artin-Schreier covers that are isomorphic as curves over $\F_q$. Then $f(x) = u g(\gamma x)$ for some $u \in \Z/p\Z^{\times}$ and $\gamma \in \PGL_2(\F_q)$.
% 	\end{claim}

	\begin{proof}
	 
	   The map $\Tilde{\phi}$ is induced by some $\gamma \in \PGL_2(\F_q)$. Let $D_f$ and $D_g$ denote the ramification divisors of $\mcC_f$ and $\mcC_g$ respectively. By Artin-Schreier theory, these are determined by the poles of $f$ and $g$ respectively. Note that since the curves are defined over $\F_q$, so are their ramification divisors. Since $\phi$ must preserve the ramification invariants (namely, the number of ramified points and the ramification groups at each of these points), we must have that $\Tilde{\phi}^*(D_g) = D_f$. Thus $\mcC_{f\circ \gamma}$ and $\mcC_{g}$ are isomorphic curves with the same ramification divisor. \\
	    
	    Now, two Artin-Schreier covers,
	    $$
	    y^p - y = f_1(x) \quad \text{and } \quad y^p - y = f_2(x)
	    $$
	    with the same genus and ramification divisor are isomorphic if and only if $f_1(x) = u f_2(x) + \delta^p - \delta$ (see, for example \cite{PriesZhu12}, Remark 3.9) with $u \in \Z/p\Z^{\times}$ and $\delta \in \F_q(x)$. Since we have imposed the condition that $f(x), g(x) \in \mathcal{S}$, the proposition follows.
	\end{proof}

  Thus we have that for $g \ge 2$, the map on the genus $g$ part in map (\ref{eqn:ArtinSchreierfibermap}) has fibers of size bounded by $C(q)g^4$, where $C(q)$ is a constant independent of $g$. For notational convenience, let $\mathscr{G} = \cup_{g \ge 0}\mathscr{AS}_g(\F_q)$.\\
   
   Since $|\{ f \in \mathcal{S} : g(\mcC_f) = g \}| \le C(q)g^4 |\{ \mcC \in \mathscr{AS}_g(\F_q) \}|$, therefore the probability
   \begin{align*}
        P(\mathscr{G}, 0) &= \lim_{X \rightarrow \infty} \frac{\mid \{ \mcC \in \mathscr{AS}_g(\F_q) \mid q^g < X, a(\mcC) = 0\} \mid }{\mid \{ \mcC \in \mathscr{AS}_g(\F_q) \mid q^g < X \} \mid }\\
        &\le \lim_{X \rightarrow \infty} C(q) \log_q(X)^4 \frac{\mid \{ f \in \mathcal{S} \mid q^{g(\mcC_f)} < X, a(\mcC_f) = 0 \}\mid }{ \mid \{ f \in \mathcal{S} \mid q^{g(\mcC_f)} < X  \} \mid  }.
   \end{align*}
    
    The counting arguments in Section \ref{subsection:mainresultsartinschreier} will show that for $p\ge 7$, the quantities $P(\mathscr{G},0)$ and $P(\Fs,0)$ are the same and equal to 0. 

	\subsection{Superelliptic curves}
    \label{superellipticbackground}
    Let $k$ be an arbitrary field.  A \emph{superelliptic curve} over $k$ is a curve defined by the affine equation,
    $$
    y^n = f(x)
    $$
    
    where $f(x) \in k[x]$ and $n$ is coprime to the characteristic of $k$. Then, this curve has an action of $\boldsymbol{\mu}_n$ ($n$-th roots of unity) on it, namely the map
    $$
    (x,y) \mapsto (x, \zeta_n y)
    $$
    
    where $\zeta_n$ is a primitive $n$-th root of unity. One can make a transformation to write 
    \begin{equation}
        f(x) = \prod_{i=1}^{n-1} (f_i(x))^i
    \end{equation}
    
    where each $f_i(x)$ is a squarefree polynomial. The quotient $\mcC/\boldsymbol{\mu}_n$ gives a map to $\PP^1$, sending $(x,y) \mapsto x$. We let $N:= \sum_{i=1}^{n-1} i\deg(f_i)$. Then the curve $\mathcal{C}$ is unramified over $\infty \in \PP^1$ if and only if $N \equiv 0 \mod n$. In the case that $N \not\equiv 0 \mod n$, we let $n_{\infty}$ be the smallest positive integer such that $N + n_{\infty} \equiv 0 \mod n$. \\
    
    Now, the map $\mathcal{C} \rightarrow \PP^1$ is ramified at the zeros of $f$ and possibly at $\infty$. The ramification indices at each of the ramified points $\alpha \in \PP^1(\bar{k})$ are given by \cite{Koo91}:
    \begin{displaymath}
        e(\alpha) = \begin{cases} \frac{n}{(n,i)} \qquad &if \; f_i(\alpha) = 0\\
        \frac{n}{(n,n_{\infty})} \qquad &if \; \alpha = \infty
        \end{cases}
    \end{displaymath}
	
	The genus of this curve is given by
	\begin{equation}
	   	g = -n+1 + \frac{1}{2}\sum_{i=1}^{n-1} \deg(f_i) (n-(n,i)) + \frac{1}{2}\epsilon (n -  (n,n_{\infty})),
	\end{equation}

	where $\epsilon$ is 0 if the map $\mcC \rightarrow \PP^1$ is unramified over $\infty$ and 1 otherwise.
	
	\begin{remark}
	    Since the techniques of this paper are based on counting polynomials, it is necessary to separate the case when the map is ramified over $\infty \in \PP^1$, even though that seems unnatural. 
	\end{remark}

	We now specialize to the case where $n$ is an odd prime. Let $B \subset \PP^1(\bar{k})$ be the set of points ramified in the cover $y^n = f(x)$. Let $\mathbf{m} = \mid \! B \! \mid $. If $\epsilon =0$, then $\mathbf{m} = \sum_{i=1}^{n-1} \deg(f_i) $ and if $\epsilon = 1$, then $\mathbf{m} = \sum_{i=1}^{n-1} \deg(f_i) +1. $
	
	In either case, we have,
	\begin{equation}
	   	g = \frac{1}{2}(n-1)(\mathbf{m} -2).
	\end{equation}

	Thus with regard to superelliptic curves, we will be interested in the family $\Fs$ of covers $y^n = f(x)$, where,
	\begin{itemize}
		\item $n$ is a fixed odd prime.
		\item The curve is defined over $\F_q$, where $q$ is a power of $2$. 
		\item $f(x) \in \F_q[x]$ is $n$-th power-free.
	\end{itemize}

	\subsubsection{$a$-number of Superelliptic curves in characteristic 2}
	
	We now give a combinatorial criterion for the ordinarity for superelliptic curves in characteristic 2. The discussion in this section is based on a paper by Elkin \cite{Elkin11}. 
	Let $\mathcal{C}$ be a smooth proper superelliptic curve over $\F_q$, $q$ a power of 2, with affine model: $ y^n = f(x)$, where $n$ is an odd prime. We maintain the same notation as before. The space $H^0(\mathcal{C}, \Omega^1_{\mathcal{C}})$ inherits the action of $\boldsymbol{\mu}_n$ and decomposes into eigenspaces as follows:
	$$
	H^0(\mathcal{C}, \Omega^1_{\mathcal{C}}) = \oplus_{i=1}^{n-1} \mathcal{D}_i.
	$$
	
	A key player in Elkin's work is the Cartier operator, $\Cs$. This is a $\Frob^{-1}$-linear operator on $H^0(\mathcal{C}, \Omega^1_{\mathcal{C}})$, which annihilates exact differentials and preserves logarithmic differentials. It is well known that the $a$-number, $a(\mathcal{C})$, equals $g(\mathcal{C}) - \rank(\Cs)$. To state the result in Elkin's paper, we first describe some notation. Let $d_i = \dim(\mathcal{D}_i)$. Let $\sigma$ be the permutation of $\{1,2, \ldots n-1\}$ defined by
	$$
	p\sigma(i) \equiv i \! \! \mod n.
	$$

	By bounding the rank of the Cartier operator, Elkin proves the following:
	
	\begin{proposition}
		\label{mainbase}
		Let $\mathcal{C}$ be as above. Then:
		$$
		g(\mathcal{C}) - a(\mathcal{C}) = \sum_{i=1}^{n-1} \min(d_i, d_{\sigma(i)})
		$$
		where the $d_i = \dim(\mathcal{D}_i)$ can be computed explicitly from the ramification invariants of the curve and $\sigma$ is the permutation of the set $\{1,2, \cdots n-1 \}$ described above.
	\end{proposition}

	For any rational number $r$, let $\langle r \rangle = r - \lfloor r \rfloor$. Elkin proves that the $d_i$'s are given by the following formula:
	\begin{equation}
	\label{dis}
	   d_i = \sum_{j=1}^{n-1} \deg(f_j) \left\langle \frac{ij}{n} \right\rangle + \left\langle \frac{in_{\infty}}{n} \right\rangle  -1 
	\end{equation}

	Recall that the ordinarity of an abelian variety is equivalent to the condition that its $a$-number is 0. Proposition \ref{mainbase} tells us that $a(\mathcal{C}) = 0$ implies that $g(\mathcal{C}) = \sum_{i=1}^{n-1} \min(d_i, d_{\sigma(i)})$. We now give a condition for ordinarity in terms of the degrees of $f_i$. We will treat the case $n=3$ separately from the case of a general odd prime.
	
	\subsubsection{The case $n=3$:} In this subsection, we consider curves of the form $\mathcal{C}: y^3 = f(x)$. The equation for the genus simplifies to:
	$$
	g = \mathbf{m} - 2.
	$$
	
	\begin{proposition}
	\label{superellordinary3}
	    A curve of the form $y^3 = f_1 f_2^2$, with $f_1, f_2$ squarefree is ordinary if and only if one of the following is true:
	    \begin{enumerate}
	        \item $n_{\infty}=0$ and $\deg(f_1) = \deg(f_2)$, or
	        \item $n_{\infty} = i$ for some $i \in \{1,2\}$ and $\deg(f_i)+1 = \deg(f_{3-i})$
	    \end{enumerate}
	\end{proposition}
	
	\begin{proof}
	    Note that $\sigma = (1 \; 2)$. So, $g = 2\min(d_1,d_2)$, which in turn implies $g = 2d_1$ or $g(\mathcal{C}) = 2d_2$. We prove case $(1)$ here. The other cases follow by a similar calculation.\\
	    
	    In this case,
	    $$
    	d_1 = \frac{1}{3} \deg(f_1) + \frac{2}{3} \deg(f_2) -1
    	$$
	
    	and
    	$$
    	d_2 = \frac{2}{3} \deg(f_1) + \frac{1}{3} \deg(f_2) -1.
    	$$
	
     Therefore, $\deg(f_1) = \deg(f_2)$. For case $(2)$, we just replace $\deg(f_i)$ by $\deg(f_i)+1$ in the expression for each $d_j$.
	\end{proof}
	
    \subsubsection{The case of a general odd prime}
	
	\begin{proposition}
	    \label{superellordinary}
	    A curve defined by $y^n = \prod_{i=1}^{n-1} (f_i(x))^i$ as in section \ref{superellipticbackground} with $n$ an odd prime, is ordinary if and only if one of the following is true:
	    \begin{enumerate}
	        \item $n_{\infty} = 0$ and $\deg(f_i) = \deg(f_{n-i})$, or
	        \item $n_{\infty} = i$ for some $i \in \{1, 2 \ldots n-1 \}$, and $\deg(f_i) +1 = \deg(f_{n-i})$, and for $j \neq i, n-i$, $\deg(f_j) = \deg(f_{n-j})$.
	    \end{enumerate}
	\end{proposition}

	\begin{proof}
	    As before, we only prove case $(1)$ and the other case follows from a modified, but similar calculation. The condition for ordinarity gives: $\sum_i d_i = \sum_i \min(d_i, d_{\sigma(i)})$. This automatically implies that $d_i = d_j$ for all $1 \le i,j \le n$. Since we are considering the case where $n_{\infty} = 0$, 
	    $$
	    d_i = \sum_{j=1}^{n-1}  \deg(f_j)\left\langle \frac{ij}{n} \right\rangle -1 .
	    $$
	    Define the matrix $A$, with $A_{ij} = \left\langle \frac{ij}{n} \right\rangle$. Thus, the degrees of $f_i$'s must satisfy:
	    \begin{equation}
	        A \begin{pmatrix} x_1 \\ x_2 \\ \vdots \\ x_{n-1}
	        \end{pmatrix}
	        = \begin{pmatrix} d+1\\ d+1 \\ \vdots \\ d+1 \end{pmatrix}
	    \end{equation}

	    for some $d \ge 0$. Let $V$ denote the space of $n-1 \times 1$ vectors whose coordinates are all equal. We are interested in (the integral points of) the space of $x = (x_1, x_2 \ldots x_{n-1})^T$ such that $Ax \in V$.
	    
	    %We claim that the solutions of this system are determined by the kernel of $A$. Indeed, since the row sum of $A$ is $\frac{n-1}{2}$, if $Ax_0 = 0$, then $A(x_0 + B_{\frac{2(d+1)}{n-1}}) = B_{d+1}$.
	    
	    \begin{lemma}
	    The space $\{x\in \Z^{n-1} \mid Ax \in V  \}$ consists of vectors $x$ for which $x_k = x_{n-k}$ for all $k=1, 2 \ldots n-1$. 
	    \end{lemma}
	
	    \begin{prooflemma}
        
        We prove this lemma by constructing an explicit basis for the kernel of $A$, $\Ker(A)$. Let $x^{(k)}$ denote the $n-1 \times 1$ vector which has $1$'s in the $k$th and $n-k$th positions and $-1$'s in the $\frac{n-1}{2}$th and $\frac{n+1}{2}$th positions. We claim that $\{x^{(k)} \mid k = 1, 2, \ldots \frac{n-3}{2} \}$ is a basis for $\Ker(A)$.
        
	    \begin{align*}
	        (Ax^{(k)})_i &= \left(\frac{ik}{n} - \leftfloor \frac{ik}{n}\rightfloor \right) + \left(\frac{i(n-k)}{n} - \leftfloor\frac{i(n-k)}{n}\rightfloor \right) - \left(\frac{i(n-1)}{2n} - \leftfloor \frac{i(n-1)}{2n}\right] \rightfloor \\ &- \left(\frac{i(n+1)}{2n} - \leftfloor\frac{i(n+1)}{2n}\rightfloor \right)\\
	        &= \leftfloor \frac{i(n-1)}{2n}\rightfloor + \leftfloor \frac{i(n+1)}{2n}\rightfloor - \leftfloor \frac{ik}{n}\rightfloor - \leftfloor \frac{i(n-k)}{n}\rightfloor \\
	        &=0.
	    \end{align*}
	    
	    Thus, it only remains to prove that $A$ has rank at least $\frac{n+1}{2}$. Now, $nA$ can be row reduced such that the top left $\frac{n+1}{2} \times \frac{n+1}{2}$ submatrix looks like:
	    \begin{align*}
	        \begin{pmatrix}
	            1 & 2 & 3 &\ldots &\frac{n-1}{2} & \frac{n+1}{2} \\
	            0 & 0 & 0  &\ldots & 0 & * \\
	            0 & 0 & 0 & \ldots & * & *\\
	            \vdots \\
	            0 & 0 & * & \ldots & * & * \\
	            0 & * & * & \ldots &* & *\\
	        \end{pmatrix}
	    \end{align*}
	    
	    where each of the entries immediately below the anti-diagonal is necessarily non zero. Such a matrix has non-zero determinant.

	    Thus, any element in the $\Ker(A)$ looks like:
	    \begin{displaymath}
	        (x_1, x_2, \ldots -\sum_{i=1}^{\frac{n-3}{2}} x_i,-\sum_{i=1}^{\frac{n-3}{2}} x_i, \ldots x_{2}, x_{1})^T
	    \end{displaymath}
	    
	    This proves the lemma and hence the proposition.
	    
	    \end{prooflemma}

	\end{proof}
    
    \begin{remark}
    \label{ramifiedtodegreei}
        Perhaps a more natural way to interpret Propositions \ref{superellordinary3} and \ref{superellordinary} is to say that for a curve $y^n = f(x)$ ($n$ prime) has ordinary Jacobian if and only if the same number of points are ramified to degree $i$ and $n-i$ for any $i \in \{1, 2 \ldots n-1 \}$. Here we say that a point $P$ is `ramified to degree $i$' if the curve locally looks like $y^n = u x_P^i$, where $x_P$ is a uniformizer at $P$ and $u$ is  unit. 
    \end{remark}

	\subsubsection{Aside on counting curves versus counting covers} One might wonder, as in the Artin-Schreier case in \S \ref{ArtinSchreierBackground}, what the difference is between counting superelliptic curves and covers of the form $y^n = f(x)$. We choose to restrict our attention to \emph{covers}, i.e. to \emph{equations} of the form $y^n = f(x)$ with $f(x) \in \F_q[x]$ $n$-th power-free, and make the claim that this does not significantly affect our results.\\

	We first introduce some notation for this section alone. For any $u \in \Z/n\Z ^{\times}$, let $[u]$ be the map that takes $\prod_i (f_i(x))^i$ to $\prod_i (f_i(x))^{(ui \! \mod n)}$. By a straightforward sequence of transformations, one can see that if $f_i$ is squarefree for each $i$ the two curves,
	$$
	y^n = \prod_i (f_i(x))^i \qquad \text{ and } \qquad y^n = \prod_i (f_i(x))^{(ui \! \mod n)}
	$$
	
	are indeed isomorphic. By abuse of notation, we also call this isomorphism of curves $[u]$. We claim that up to an automorphism of $\PP^1_{\F_q}$, the only isomorphisms between superelliptic covers are of the form $[u]$, with $u \in \Z/n\Z$. This is a standard Kummer theory argument, whose proof we recall here. Let $\zeta_n$ be an $n$-th root of unity that acts as an automorphism of the curve sending $(x,y) \mapsto (x, \zeta_n y)$.
	
	\begin{claim}
	    For $n$ an odd prime, let $f(x) = \prod_{i=1}^{n-1} (f_i(x))^i$ and $g(x) = \prod_{i=1}^{n-1} (g_i(x))^i$ be two monic $n$-th power-free polynomials in $\F_q[x]$ such that:
	    \begin{itemize}
	        \item For each $i$, $f_i(x)$ and $g_i(x)$ are squarefree,
	        \item $\Div_0(f) = \Div_0(g)$.
	    \end{itemize}
	    
	    Suppose that $\mcC_f: y^n = f(x)$ and $\mcC_g: y^n = g(x)$ are isomorphic as curves via an isomorphism $\phi$. Then there is a $u \in \Z/n\Z^{\times}$ such that $\phi = \zeta_n \circ [u]$.
	\end{claim}

	\begin{proof}
	     Let $K = \F_q(\PP^1)$ and $L = \F_q(\mcC_f) \cong \F_q(\mcC_g)$. Note $L(\zeta_n)/K(\zeta_n) $ is a Galois extension. Let $\varphi : \Gal(\overline{K(\zeta_n)}/K(\zeta_n)) \rightarrow \boldsymbol{\mu}_n$ be the homomorphism corresponding to $L(\zeta_n)$. Any other field $L^{'}$ that is isomorphic to $L(\zeta_n)$ corresponds to the homomorphism $\varphi^u$ for some $u \in \Z/n\Z^{\times}$. Therefore, if $[\alpha] \in K(\zeta_n)^{\times}/(K(\zeta_n)^{\times})^n$ is the class corresponding to $\varphi$ via the Kummer map, then there is a $u \in \Z/n\Z^{\times}$ such that the isomorphism $\F_q(\zeta_n)(\mcC_f) \cong \F_q(\zeta_n)(\mcC_g)$ corresponds to the class $[\alpha^u]$. This proves the claim.
	\end{proof}

%	    Change: $y \mapsto 1/y$ to get:
%	    $$
%	    \frac{1}{y^3} = x(x+1)^2(x^3 + x+ 1)
%	    $$ 
	    
%	    so
	    
%	    $$
%	    \frac{1}{x(x+1)^2(x^3+ x + 1)} = y^3
%	    $$
	    
%	    multiply both sides by $x^3(x+1)^3(x^3+x+1)^3$ and transform again to get what you want.
    For $n$ an odd prime, let $\mathcal{T}_n$ denote the set of $n$-th power free polynomials in $\F_q[x]$. Let $\mathscr{SE}_{g,n}(\F_q)$ denote the set of superelliptic curves of degree $n$ and genus $g$ over $\F_q$. Then the above claim shows that the fibers of the map,
    \begin{align*}
        \mathcal{T}_n &\rightarrow \bigcup_{g \ge 0} \mathscr{SE}_{g,n}(\F_q)\\
        f(x) & \mapsto (y^n = f(x))
    \end{align*}
    
    have size bounded by $n\mid \! \Z/n\Z^{\times} \! \mid \mid \PGL_2(\F_q) \mid$. 
	%%%%%%%%%%%%%%%%
	In particular, the probability
     $$
     \lim_{X \rightarrow \infty} \frac{\mid \{ \mcC \in \mathscr{SE}_{g,n}(\F_q) \mid q^g < X, a(\mcC) = 0\} \mid }{\mid \{ \mcC \in \mathscr{SE}_g(\F_q) \mid q^g < X \} \mid }
     $$
    
     exists and is zero if and only if the quantity
    \begin{align}
     \label{replacementqty}
         \lim_{X \rightarrow \infty} \frac{\mid \{ f \in \mathcal{T}_n \mid q^{g(\mcC_f)} < X, a(\mcC_f) = 0 \}\mid }{ \mid \{ f \in \mathcal{T}_n \mid q^{g(\mcC_f)} < X  \} \mid  }
     \end{align}
    
     exists and is zero. This proves that understanding the proportion of ordinarity in the set of superelliptic curves of degree $n$ is the same as understanding it for the family of superelliptic curves of a fixed degree over $\F_q$.

	%%%%%%%%%%%%%%%

    	\section{Main Results}
	\label{section:mainresults}
	In this section we describe the main results obtained from counting each of the families described above. Our main tool will be the following Tauberian theorem:
	
	\begin{theorem}
		[See \cite{TschinLoir01}, Appendix A]
		\label{Tauberian}
		Let $\{\lambda_n\}_{n \in \Z_{>0}}$ be strictly increasing sequence of positive integers. Let $f$ be the Dirichlet series,
		$$
		f(s) = \sum_{n=1}^{\infty} c_n \lambda_n^{-s}.
		$$
		Further, assume the following:
		\begin{enumerate}
			\item $f(s)$ converges for $Re(s) > a > 0$.
			\item $f$ admits a meromorphic continuation to $Re(s) > a - \delta_0 > 0$ for some $\delta_0>0$.
			\item The right-most pole of $f$ is at $s=a$, with multiplicity $b \in \mathbb{N}$. Let $\Theta = \lim_{s \rightarrow a} f(s)(s-a)^b$.
			\item (Technical assumption) $\exists $ a $\kappa > 0$ such that for $Re(s) > a - \delta_0$,
			$$
			\left| \frac{f(s)(s-a)^b}{s^b} \right| = O((1 + Im(s))^{\kappa})
			$$
			
			Then there exists a (monic) polynomial $P$ of degree $b-1$ such that for any $\delta < \delta_0$, we have,
			$$
			\sum_{\lambda_n < X} c_n = \frac{\Theta}{a (b-1)!} X^a P(\log(X)) + O(X^{a-\delta}).
			$$
		\end{enumerate}
	\end{theorem}

	We will henceforth use the notation $\barsQ$ to denote $q^{\deg(Q)}$, where $Q$ is an irreducible polynomial over $\F_q$. We will denote by $\zeta(s)$, the zeta function of $\mathbb{A}^1_{\F_q}$. Thus $\zeta(s) = \prod_Q (1 - \barsQ^{-s})^{-1}$, where the product is over monic irreducible polynomials over $\F_q$.
	
	\subsection{Artin-Schreier curves}
	\label{subsection:mainresultsartinschreier}
	To recall, the family $\Fs$ that we are interested in in this section is that of covers $y^p - y = f(x)$, with $f(x) \in \mathcal{S}$, such that the corresponding map $\mcC \rightarrow \PP^1$ is unramified over $\infty$.\\

	We first set up some notation in order to calculate $N(\Fs,X)$ and $N(\Fs,0,X)$:
	\begin{itemize}
	    \item Define a new invariant: $m = \frac{2g}{p-1} +2$. By equation \eqref{genusArtinSchreier}, this is an integer and is equal to
	    $$
	    \sum_{Q} \deg(Q)(d_{Q}+1).
	    $$
	    \item For any $m \ge 2$, let $a(m)$ be the number of Artin-Schreier covers $\mcC$ with the above invariant equal to $m$. Let $b(m)$ be the number of such covers with $a$-number 0.
	    \item Define 
	    $$
	    N^*(\Fs,X) = \sum_{q^m < X}a(m)  \qquad \text{and} \qquad N^*(\Fs,0,X) = \sum_{q^m < X}b(m)
	    $$
	    We will calculate these as an intermediate step towards finding $N(\Fs,X) = N^*(\Fs, q^2X^{2/(p-1)})$ and $N(\Fs,0,X) = N^*(\Fs,0, q^2X^{2/(p-1)})$.
	\end{itemize}
	
	For this section, define the zeta function,
	$$
	Z(s) = \sum_{\mcC \in \Fs} q^{-m(\mcC)s} = \sum_m a(m)q^{-ms}.
	$$
	
	\begin{lemma}
	\label{lemmaASmain}
	     $Z(s)$ converges for $Re(s)>1$ and has a pole of order $p-1$ at $s=1$.
	\end{lemma}
	
	\begin{proof}
	    Note that,
	    $$
	    m = \sum_{Q} \deg(Q)(d_{Q}+1)
	    $$
	    where the sum is over monic irreducible polynomials $Q \in \F_q[x]$. Here, we set $d_{Q}=-1$ if the map $\mcC \rightarrow \PP^1$ is unramified over the divisor $\Div(Q)$. 
	    Since this is a sum of local factors, we factor $Z(s)$ as a product of local functions, i.e. $Z(s) = \prod_{Q} Z_Q(s)$, where $Q$ varies over monic irreducible polynomials in $\F_q[x]$. We can write $Z_Q(s) = \sum_{k \ge 0} c(k)\barsQ^{-ks}$. Recall from \S \ref{ArtinSchreierBackground} that if $\alpha \in B$ and $Q(\alpha) = 0$, then $d_{Q} = \deg(f_{\alpha})$ as in the partial fraction decomposition of $f(x)$ \eqref{partialdecomp}. Further recall that in each $f_{\alpha}$, the coefficient of $x^{ip}$ is 0 for each $0 \le i \le \lfloor d_Q/p \rfloor$. Since $k = d_Q+1$, 
	    $$
	    c(k) = \#\{ f_{\alpha} \in \F_{\barsQ} \mid \deg(f_{\alpha}) = k-1, \text{ coefficient of } x^{ip} = 0 \}
	    $$
	    
	    where $k \not\equiv 1 \mod p$ (since $d_{Q} \not\equiv 0 \mod p$). We write $d_{Q} = np + i$, with $1 \le i \le p-1$. The above discussion gives us that for $k = np+i+1$,
	    $$
	    c(k) = (\barsQ -1)\barsQ^{i-1} \barsQ^{n(p-1)}
	    $$
	    
	    For convenience, we distinguish the cases where $p=2$ and $p\ge 3$.\\
	    
	    For $p=2$,
	    \begin{align*}
	        Z_{Q}(s) &= 1 + \sum_{n=0}^{\infty} (\barsQ-1)\barsQ^n\barsQ^{-s(2n+2)} \\
	        &=\frac{1-\barsQ^{-2s}}{1-\barsQ^{1-2s}}.
	    \end{align*}
	    
	    For $p\ge 3$,
	    \begin{align*}
	        Z_Q(s) &= 1 + \sum_{i=1}^{p-1} \sum_{n = 0}^{\infty} (\barsQ -1)\barsQ^{i-1} \barsQ^{n(p-1)} \barsQ^{-s(np+i+1)}\\
	        %&= 1 +  (\barsQ-1)\barsQ^{-2s} \sum_{i=0}^{p-2} \barsQ^{i(1-s)} \sum_{n = 0}^{\infty} \barsQ^{n(p-1 - ps)}\\
	        &= 1 + \left(\frac{(\barsQ-1)\barsQ^{-2s}}{1 - \barsQ^{p-1-ps} }\right) \sum_{i=1}^{p-1} \barsQ^{i(1-s)}\\
	       % &= \frac{1}{1 - \barsQ^{p-1-ps}} \left(1 - \barsQ^{p-1-ps} + (\barsQ^{1-2s} - \barsQ^{-2s}) \sum_{i=0}^{p-2} \barsQ^{i(1-s)} \right)\\
	        %&= \frac{1 - \barsQ^{p-1-ps} + \sum_{i=0}^{p-2} \barsQ^{(i+1)-(i+2)s} - \sum_{i=0}^{p-2} \barsQ^{i-(i+2)s}}{1 - \barsQ^{p-1-ps}}\\
	        &= \frac{1  + \sum_{i=0}^{p-3} \barsQ^{(i+1)-(i+2)s} - \sum_{i=0}^{p-2} \barsQ^{i-(i+2)s}}{1 - \barsQ^{p-1-ps}}.\\
	    \end{align*}

	    For $p \ge 3$, let
	    $$
	    \psi_{p,Q}(s) = \left( 1  + \sum_{i=0}^{p-3} \barsQ^{(i+1)-(i+2)s} - \sum_{i=0}^{p-2} \barsQ^{i-(i+2)s} \right)\prod_{i=0}^{p-3}(1- \barsQ^{(i+1)-(i+2)s}).
	    $$
	    
	    Define
	    \begin{align}
	    \psi_p(s) = \begin{cases} \zeta(2s)^{-1} \qquad &if \; p=2\\ \prod_{Q} \psi_{p,Q}(s) \qquad &if \; p \ge 3.
	    \end{cases}
	    \end{align}
	    
	    Then by a straightforward calculation, $\ds \psi_p(s)$ converges for $Re(s)> \delta_p$, for some constant $\delta_p \in (0,1)$. \\

	    Therefore,
	    \begin{align*}
	    \prod_{Q} Z_Q(s) = \psi_p(s) \prod_{i=0}^{p-2} \zeta(s(i+2)-(i+1)).
	    \end{align*}
	    
	    This converges for $Re(s)>1$ and has a pole of order $p-1$ at $s=1$. Further, the residue at $s=1$ is given by
	    $$
	    \lim_{s \rightarrow 1} Z(s)(s-1)^{p-1} = \frac{\psi_p(1)}{\log(q)^{p-1}}.
	    $$
	\end{proof}
	
    To count the number of ordinary curves, we define
	$$
	Z_0(s) := \sum_{\mcC \in \Fs_0} q^{-m(\mcC)s} = \sum_m b(m) q^{-ms}.
	$$
	
	Recall that for such curves, $d_{\alpha}=1$ for all $\alpha$. Therefore, the $Z_0(s)= \prod_Q Z_{0,Q}(s)$, where the local factors are
	$$
	Z_{0,Q}(s) = 1+ (\barsQ-1)\barsQ^{-2s}.
	$$
	
	\begin{lemma}
	\label{lemmaASordinary}
	    $Z_{0}(s)$ converges for $Re(s) > 1$ and has a simple pole at $s=1$.
	\end{lemma}
	
	\begin{proof}
	    Note that 
	    
	    $$
	    (1+ \barsQ^{1-2s} - \barsQ^{-2s})(1-\barsQ^{1-2s}) = 1 - \barsQ^{-2s} - \barsQ^{2-4s}+ \barsQ^{1-4s} 
	    $$
	    
	    and
    	$$
    	\phi(s) := \prod_Q (1 - \barsQ^{-2s} - \barsQ^{2-4s}+ \barsQ^{1-4s})
	    $$
	
	    converges for $Re(s)> 3/4$. Therefore $Z_{0}(s) = \phi(s) \zeta(2s-1)$ converges for $Re(s)>1$ and has a simple pole at $s=1$. Further, the residue at $s=1$ is
    	$$
    	\lim_{s \rightarrow 1} Z_0(s)(s-1) =  \frac{\phi(1)}{\log(q)}.
    	$$
	\end{proof}

	\begin{proposition}
	    For any $\delta >0$,
	    \begin{align*}
	        N^*(\Fs,X) &= \frac{\psi_p(1)}{\log(q)} X(\log_q(X))^{p-2} + O(X^{1-\delta}) 
           \\
            N^*(\Fs,0,X) &= \frac{\phi(1)}{\log(q)}X + O(X^{1-\delta})
	    \end{align*}
	\end{proposition}
	
	\begin{proof}
	    This follows from the Tauberian theorem \ref{Tauberian} applied to the results of Lemmas \ref{lemmaASmain} and \ref{lemmaASordinary}, since $\zeta(s)$ has a meromorphic continuation to the entire complex plane.  
	\end{proof}

	\begin{corollary}
	\label{ArtinSchreierMainTheorem}
	    For any $\delta >0$,
	     \begin{align*}
	        N(\Fs,X) &= \frac{\psi_p(1)}{\log(q)} q^2X^{2/(p-1)}(\log_q(X^{2/(p-1)}))^{p-2} + O(X^{\frac{2}{p-1} - \delta}),
           \\
            N(\Fs,0,X) &= \frac{\phi(1)}{\log(q)}q^2X^{2/(p-1)} + O(X^{\frac{2}{p-1} - \delta}).
	    \end{align*} 
	    
	    In particular, the probability that an Artin-Schreier cover unramified over $\infty$ is ordinary, is
	    \begin{align*}
	        \phi(1) \zeta(2) \qquad &if\; p=2\\
	        0 \qquad & if \; p\ge 3.
	    \end{align*}
	    
	\end{corollary}
	
	\begin{proof}
	    This follows from the fact that $N(\Fs,X) = N^*(\Fs, q^{2}X^{2/(p-1)})$. 
	    
	\end{proof}
	
	\begin{remark}
	\label{RemarkArtinSchreier}
	   We conclude our section on counting Artin-Schreier curves with some remarks.
	   \begin{enumerate}
	       %\item Calculating the first few terms of the product $\phi(1) \zeta(2)$ gives
	       %$$
	       %\phi(1)\zeta(2) = 1 - q^{-1} + q^{-2} - 2q^{-3} + O(q^{-4}).
	       %$$
	       
	   \item If we modify $\Fs$ to include the covers ramified over $\infty$, we must modify the partial fraction decomposition in \eqref{partialdecomp} to:
	   $$
	   f(x) = \sum_{\alpha \in B \atop \alpha \neq \infty} f_{\alpha}(x_{\alpha}) + g(x).
	   $$
	   
	   Here $g(x) \in \F_q[x] $ is a polynomial that, like the other $f_{\alpha}$'s, has degree coprime to $p$ and the coefficients of $x^{ip}$ in $g(x)$ are 0, for all $0 \le i \le \lfloor \deg(g)/p \rfloor$. This manifests as a change in the zeta functions $Z(s)$ and $Z_0(s)$ defined in the above discussion by factors that we will call $Z_{\infty}(s)$ and $Z_{0, \infty}(s)$ respectively. That is, we write $Z(s) = Z_{\infty}(s) \prod_Q Z_Q(s)$ and  $Z_0(s) = Z_{0,\infty}(s) \prod_Q Z_{0,Q}(s)$. Both these factors only affect the residues of $Z(s)$ and $Z_0(s)$, which means that for $p \ge 3$, the probability of ordinarity for the modified family is still 0. For $p=2$,
	   $$
	   Z_{\infty}(s) = 1+q^{-1} \qquad \text{and} \qquad Z_{0,\infty}(s) = 1-q^{-1} + q^{-2}.
	   $$
	   Therefore the probability of ordinarity in the modified family is
	  \begin{equation}
	      \left(\frac{1-q^{-1}+q^{-2}}{1+q^{-1}} \right) \phi(1) \zeta(2) = 1 - 3q^{-1} + 6q^{-2} + O(q^{-3}).
	  \end{equation}
	  Some values of the above probability are calculated in Table \ref{Table1} in \S \ref{section:data}.
	  
	  \item Recall from Section \ref{subsubsection:AScurvesversuscovers}, that the probability that an Artin-Schreier \emph{curve} is ordinary, is bounded above by the  quantity,
	 $$
	 \lim_{X \rightarrow \infty} C(q)\log_q(X)^4 \frac{N(\Fs,g,X)}{N(\Fs,X)}.
	 $$
	 Since the order of growth of $N(\Fs,g,X)$ is $X^{2/(p-1)}$, and the that of $N(\Fs,X)$ is $X^{2/(p-1)}\log(X)^{p-2}$, this quantity is $0$ whenever $p\ge 7$. The geometric description of the Artin-Schreier locus leads us to believe that the same might be true for $p=3,5$, as explained in \S \ref{subsubsection:irreducibility} below. However a proof for these cases requires more work.
	
	\item Recall that if the Jacobian of a curve behaves randomly in the sense of \cite{CEZB13}, the heuristics predict that the probability of that a curve is ordinary is:
	$$
	\prod_{i=1}^{\infty} (1+ q^{-i})^{-1}.
	$$
	
	Corollary \ref{ArtinSchreierMainTheorem} and the above remark prove that the Jacobian of an Artin-Schreier curve does not behave randomly in the sense of \cite{CEZB13}. For $p \ge 3$ this is clear. For $p=2$, elementary calculations show that the constants are not equal. In fact, 
	$$
	\prod_{i=1}^{\infty} (1+ q^{-i})^{-1} = 1 - q^{-1} - q^{-3} + q^{-4} + O(q^{-5}). 
	$$

	   \end{enumerate}
	\end{remark}
	
	\subsubsection{Note about irreducibility}
	\label{subsubsection:irreducibility}
	For $p=2$ the Artin-Schreier locus $\mathscr{AS}_g$ coincides with the hyperelliptic locus $\mathscr{H}_g$. However, in general, $\mathscr{AS}_g$ is not irreducible. In \cite{PriesZhu12}, the authors give the following characterization for the irreducibility of $\mathscr{AS}_g$:
	\begin{proposition}[\cite{PriesZhu12} Corollary 1.2]
	    The moduli space $\mathscr{AS}_g$ is irreducible in exactly the following cases: (a) $p=2$, or (b) $g=0$ or $g = \frac{p-1}{2}$, or (c) $p=3$ and $g= 2, 3, 5$
	\end{proposition}
	
	In particular, for $p\ge 3$ and $g$ large enough, the Artin-Schreier locus is reducible. Further, the dimension of each irreducible component of $\mathscr{AS}_g$ is $\frac{2g}{p-1} -1$. For a given $p$-rank $r = (s-1)(p-1)$, the irreducible components of the corresponding $p$-rank stratum $\mathscr{AS}_{g,r} \subset \mathscr{AS}_g$ are in bijection with partitions of the integer $m = \frac{2g}{p-1}+2$ into $s$ integers that are $\not\equiv 1 \mod p$. If $\vec{E} =  \{e_1 \ldots e_s \}$ is one such partition, then the dimension of the corresponding Artin-Schreier stratum is $\frac{2g}{p-1} -1 - \sum_{i=1}^s \lfloor \frac{e_i-1}{p}\rfloor $ (see \cite{PriesZhu12}). 
	For example, the ordinary locus corresponds to the partition $\{2,2, \ldots 2 \}$ and has dimension $m-3$. In particular, this implies that for $p\ge3$, the proportion of components intersecting the ordinary locus goes to $0$ as $g \rightarrow \infty$. Indeed, for $p \ge 3$, let $A = \{2, 3, \ldots p\}$ and let $p_{A}(n)$ denote the number of partitions of an integer $n$ into integers from the set $A$. Then the number of components of dimension $m-3$ is $p_A(m)$. As $m \rightarrow \infty$, $p_A(m) \sim K m^{p-2}$ for some constant $K$. %{\Note do I need a reference?}	
	This might be a somewhat satisfying geometric explanation, especially for those taken aback by the fact that counting squarefree rational functions in this order gives a proportion of $0$.\\

	It is interesting to ask whether the reducibility of $\mathscr{AS}_g$ completely explains the probability obtained in Theorem \ref{ArtinSchreierMainTheorem}. That is, for each $g$, let $\overline{\mathscr{AS}}_{g,g}$ denote the closure of the ordinary locus inside $\mathscr{AS}_g$. Then, is
	\begin{equation}
	\label{equationclosure}
	    	\lim_{X \rightarrow \infty} \frac{\#\{\mcC \in \mathscr{AS}_{g,g} \mid q^g < X \}  }{\# \{ \mcC \in \overline{\mathscr{AS}}_{g,g}(\F_q) \mid q^g < X \}}
	\end{equation}

	positive?\\

	Using work of \cite{Dang18} and similar techniques as above, we can show that for $p=3$, the quantity \eqref{equationclosure} is indeed positive. This is trickier for larger primes since the combinatorial description of the closure gets more involved as $p$ increases. It would be interesting to explore this question for large primes, especially if one could find a uniform proof for all odd primes.

	\subsection{Superelliptic curves in characteristic 2}
	
	For this section, we refer back to the notation of Section \ref{superellipticbackground}. We are interested in counting covers in the family $\Fs$ of covers $y^n = f(x)$ over a field $\F_q$ of characteristic 2, where:
	\begin{itemize}
	    \item $n$ is a fixed odd prime,
	    \item $f(x) \in \F_q[x]$ is $n$-th power free.
	\end{itemize}
	
	For convenience, we count by $q^{\mathbf{m}}$ instead of $q^g$, where
	$$
	\mathbf{m} = \frac{2g}{n-1} + 2
	$$
	
	is the number of points in $\PP^1(\bar{k})$ over which the curve given by $y^n = f(x)$ is ramified. Since $n$ is fixed in the entire discussion, this will not change the method of counting significantly. Define $N^*(\Fs,X) $ as the set of curves in $\Fs$ with $q^{\mathbf{m}}<X$ and $N^*(\Fs,0,X)$ similarly. We have,
	\begin{align}
	\label{numberdefn}
	    N(\Fs, X) &= N^*(\Fs, q^{2}X^{2/(n-1)}) \quad \text{ and } \quad  N(\Fs, 0, X) = N^*(\Fs, 0, q^2X^{2/(n-1)}).
	\end{align}
	
	We define,
	 $$
	    \cF_{e_1,e_2 \cdots e_r} = \{ F_1F_2^2 \cdots F_r^r \mid F_i \in \F_q[x] \text{ monic, squarefree and mutually coprime, } \deg(F_i) = e_i \}.
	  $$
	  
	When we write $\mathbf{m} = \sum_{i=1}^{n-1} e_i$, we will be interested in the case when there are $e_i$ points ramifying to degree $i$. This is the same as the notion defined in Remark \ref{ramifiedtodegreei}. To express this concretely in terms of polynomials, it is best to use an example. For instance, for a curve given by $y^3 = F_1(x) (F_2(x))^2$, where $F_1(x) (F_2(x))^2 \in \mathcal{F}_{2,4}$, there are $2$ points that occur with degree 1 and 4 that occur with degree 2. If on the other hand, the curve is given by  $y^3 = F_1(x) (F_2(x))^2$, where $F_1(x) (F_2(x))^2 \in \mathcal{F}_{3,2}$, there are 3 points that occur with degree 1 and 3 that occur with degree 2 (since $n_{\infty} = 2$, the curve is ramified over $\infty \in \PP^1$ to degree 2).
 	
	\begin{proposition}
	    Consider the set $S_{\mathbf{m}}$ of superelliptic curves with the number of ramified points $\mathbf{m} = \sum_{i=1}^{n-1}e_i$, such that there are $e_i$ points that ramify to degree $i$. Then the size of $S_{\mathbf{m}}$ is
	    \begin{align*}
	        \mid \! \cF_{e_1,e_2 \cdots e_{n-1}} \! \mid  + \sum_{i=1}^{n-1} \mid \! \cF_{e_1, \cdots e_i-1, \cdots e_{n-1}} \! \mid .
	    \end{align*}
    \end{proposition}
    
   % We must remark here that the adjustments made to account for ramification at $\infty \in \PP^1$ arise due to the method we will use throughout the paper, which involves counting polynomials rather than divisors.
	 
	\begin{proof}
	    Let $\mcC \in \Fs$, such that $\mcC \rightarrow \PP^1$ is ramified over $\mathbf{m}$ points in $\PP^1(\overline{\F}_q)$. If the map is not ramified over $\infty$, then $\mcC \in \cF_{e_1,e_2 \ldots e_{n-1}}$. If it is ramified over $\infty$ and $n_{\infty} = i$, then $\mcC \in \cF_{e_1,\ldots e_i -1, \ldots e_{n-1}}$.  
	\end{proof}

	In the above proposition, imposing the condition $\mathbf{m} = \sum_{i=1}^{n-1}e_i$, with $e_i$ points occuring with degree $i$, implies that $\sum_{i=1}^{n-1} ie_i \equiv 0 \mod n$.
	Therefore we are interested in the quantity,
	\begin{equation}
	\label{interestsuperell}
	    \sum_{q^{\mathbf{m}} <X} \left( \mid \! \cF_{e_1,e_2 \cdots e_{n-1}} \! \mid  + \sum_{i=1}^{n-1} \mid \! \cF_{e_1, \cdots e_i-1, \cdots e_{n-1}} \! \mid \right)
	\end{equation}
	
	where the sum is over $(e_1,e_2 \ldots e_{n-1})$ such that $\sum_{i=1}^{n-1} ie_i \equiv 0 \mod n$. Further, observe that 
	\begin{align*}
	\sum_{(e_j), q^{e_1+ e_2 \ldots e_{n-1}}<X \atop  \sum ie_i \equiv 0 \mod n} \mid \! \cF_{e_1, \cdots e_i-1, \cdots e_{n-1}} \! \mid &= \sum_{(d_i), q^{d_1 + d_2 \ldots d_{n-1}}< X/q \atop \sum{id_i} \equiv i \mod n } \mid \! \cF_{d_1, \cdots d_i, \cdots d_{n-1}} \! \mid.
	\end{align*}

	Therefore \eqref{interestsuperell} can be rewritten as
	$$
	\left( \sum_{q^{e_1+e_2 \ldots e_{n-1}} <X/q} \mid \! \cF_{e_1, e_2 \ldots e_{n-1}} \! \mid  \right) + \sum_{X/q < q^{e_1 + \ldots e_{n-1}} < X \atop \sum{ie_i} \equiv 0 \mod n } \mid \! \cF_{e_1, e_2 \ldots e_{n-1}} \! \mid,
	$$
	where the first sum is over all tuples $(e_1, e_2 \ldots e_{n-1})$ with $q^{\sum e_i}<X/q$. Thus the number of superelliptic curves with $q^{\mathbf{m}}<X$ is bounded below by the quantity
	$$
	\left( \sum_{q^{e_1+e_2 \ldots e_{n-1}} <X/q} \mid \! \cF_{e_1, e_2 \ldots e_{n-1}} \! \mid  \right). 
	$$
	
	For our proof, we only need a lower bound for the total number of superelliptic curves. Therefore, it suffices to estimate this quantity. 
	
	\begin{notation}
	From now on, a sum of the form 
	$$
	\sum_{q^{e_1+ e_2 \ldots e_r}<X}
	$$
	will denote a sum over all tuples of non-negative integers $(e_1, e_2 \ldots e_r)$, with $q^{\sum e_i}<X$. 
	\end{notation}

	For any non-negative integer $m$, let $$a(m) = \ds \sum_{e_1+e_2 \ldots e_{n-1} = m} \mid \! \cF_{e_1, e_2 \ldots e_{n-1}} \! \mid. $$
	
	Define
    $$
    Z(s) = \sum_{m \ge 0} a(m) q^{-m s}.
    $$

    One way to think about an element of $\cF_{e_1, e_2 \ldots e_{n-1}}$ is to say that we are considering a polynomial
    $$
    f(x) = \prod_{i=1}^{n-1}(F_i(x))^i,
    $$
    
    such that $H := \prod_{i=1}^{n-1} F_i(x)$ is squarefree. We will use this characterization to calculate $Z(s)$.\\
    
    Now consider a squarefree polynomial $H$. Let $H = \prod h_j$ be its factorization into irreducible polynomials. We want to count the number of ways in which $H$ can be written as a product of squarefree polynomials $\prod_{i=1}^{n-1} F_i$. For each factor $h_j$, there are $n-1$ choices of squarefree polynomial that it could divide. Therefore, the number of factorizations $H = \prod_{i=1}^{n-1} F_i$ is 
    \begin{displaymath}
        (n-1)^{\omega(H)}
    \end{displaymath}

    where $\omega(H) =$ the number of distinct irreducible factors of $H$. 
    Therefore, 
    $$
    Z(s) =  \sum_{H \text{ sq. free}} (n-1)^{\omega(H)} \mid \! H \!     \mid^{-s}  = \prod_{Q} (1+ (n-1)\barsQ^{-s}).
    $$
	
	\begin{note}
	    Let $\Phi_k(s) = \prod_{Q} (1+ k\barsQ^{-s})$. Then $\Phi_k(s) \zeta(s)^{-k}$ is a function that converges for $Re(s) > 1/2$. We will denote $\Phi_k(s) \zeta(s)^{-k}$ by $\phi_k(s)$.
	\end{note}

%\begin{lemma}
%	Define $\ds \Phi_k(s) = \prod_{Q} (1+ k\barsQ^{-s})$. Then $\ds \Phi_k(s)$ has a pole of order $k$ at $s=1$.
%\end{lemma}

%\begin{proof}
	%We can write $\Phi_k(s) = \zeta(s)^k \phi_k(s)$, where $\phi_k(s)$ converges for $Re(s) >1/2$. The lemma follows from the Tauberian theorem \ref{Tauberian}. 	In particular, note that the residue of $\Phi_k(s)$ at $s=1$ is: $\frac{\phi_k(1)}{(\log(q))^k}$.
	
%	\begin{align*}
%	(1+k\barsQ^{-s})(1-\barsQ^{-s})^{k} &= (1+k\barsQ^{-s})\left( \sum_{i=0}^{k} \left(k \atop i\right)(-1)^{i}\barsQ^{-si}    \right)\\
%	&= (1 - k \barsQ^{-s} + \frac{k(k-1)}{2} \barsQ^{-2s} \cdots + k\barsQ^{-s} \cdots)\\
%	&= (1 + \barsQ^{-2s}(Poly))
%	\end{align*}
	
%	where $Poly$ is a polynomial in $\barsQ^{-s}$. 
	
%\end{proof}

\begin{proposition}
As $X \rightarrow \infty$,
	$$
	N^*(\Fs, X) \ge  \frac{\phi_{n-1}(1)}{q\log(q) (n-2)!} X(\log_q(X))^{n-2} + O(X(\log(X))^{n-3})
	$$

\end{proposition}

\begin{proof}
	Note that
	\begin{align*}
			Z(s) &= \Phi_{n-1}(s) \\
			&= \zeta (s)^{n-1} \phi_{n-1}(s).
	\end{align*}
	
	This function has a pole of order $n-1$ at $s = 1$. Thus, the Tauberian theorem implies that
	
	$$
	\sum_{m<X/q}a(m) = \frac{\phi_{n-1}(1)}{q\log(q) (n-2)!} X(\log_q(X))^{n-2} + O(X(\log(X))^{n-3}),
	$$
    which provides a lower bound for $N^*(\Fs,X)$.
\end{proof}

	\begin{corollary}
	\label{superellpropmain}
	    The number of superelliptic covers with invariant $\mathbf{m}$ such that $q^{\mathbf{m}}<X$ is bounded below by
	    $$
	    \kappa_n(q) X^{2/(n-1)} \log_q(X^{2/(n-1)})^{n-2} + O(X^{2/(n-1)} \log(X^{2/(n-1)})^{n-3}),
	    $$
	    
	    where $$
	    \kappa_n(q) =  \frac{q\phi_{n-1}(1)}{\log(q)) (n-2)!}.
	    $$
	\end{corollary}
	
	\begin{proof}
	    This follows from the fact that $N(\Fs,X) = N^*(\Fs, q^2X^{2/(n-1)})$.
	\end{proof}

	\subsubsection{Upper bounds for $N^*(\Fs,0,X)$}
	
	In this subsection, we find an upper bound for the quantity $N^*(\Fs,0,X)$, as defined in \eqref{numberdefn}. We will maintain the notation of \S \ref{superellipticbackground}.\\
	
	Suppose we consider covers with $\mathbf{m} = \sum_{i=1}^{n-1} e_i$ ramification points, $e_i$ points occurring with degree $i$. Using the criterion for ordinarity in Proposition \ref{superellordinary}, we can derive the following conditions on $\cF_{e_1, e_2 \ldots e_{n-1}}$:
	\begin{enumerate}
	    \item If $n_{\infty}=0$, $\deg(f_i) = \deg(f_{n-i})$. Note that in this case, the cover belongs to $\cF_{e_1, e_2 \ldots e_{n-1}}$ with $\deg(f_i) = e_i$. Therefore the condition for ordinarity implies that there are
	    $$
	    \mid \! \cF_{e_1, \ldots e_{\frac{n-1}{2}}, e_{\frac{n-1}{2}} \ldots e_1} \! \mid 
	    $$
	    curves of such kind over $\F_q$.
	    
	    \item If $n_{\infty} = i$, then $\mcC \in \cF_{e_1, \ldots e_i-1 \ldots e_{n-1}}$ with $\deg(f_j) = e_j$ for $j \neq i$ and $\deg(f_i) = e_i-1$. Further, the condition for ordinarity gives: for $j \neq i, n-i$, $\deg(f_j) = \deg(f_{n-j})$ and therefore $e_j = e_{n-j}$. Also, $\deg(f_i)+1 = \deg(f_{n-i})$ implies $e_i = e_{n-i}$. Therefore, the number of such curves is,
	    \begin{displaymath}
	    \mid \! \cF_{e_1, e_2 \ldots, e_i-1, \ldots e_{\frac{n-1}{2}}, e_{\frac{n-1}{2}} \ldots, e_{i} \ldots e_2, e_1} \! \mid.
	    \end{displaymath}
	    
	    if $i \le \frac{n-1}{2}$, and 
	    
	    $$
	    \mid \! \cF_{e_1, e_2 \ldots, e_i, \ldots e_{\frac{n-1}{2}}, e_{\frac{n-1}{2}} \ldots, e_{i}-1 \ldots e_2, e_1} \! \mid
	    $$
	    
	    if $i >\frac{n-1}{2}$.
	\end{enumerate}
	
	As in \eqref{interestsuperell}, we are interested in the size
    \begin{equation}
	\label{anumber0}
	    N^*(\Fs,0,X) = \sum_{q^{\mathbf{m}} <X}^* \left( \mid \! \cF_{e_1,e_2 \cdots e_{n-1}} \! \mid  + \sum_{i=1}^{n-1} \mid \! \cF_{e_1, \cdots e_i-1, \cdots e_{n-1}} \! \mid \right)
	\end{equation}
	
	where the sum is now over tuples $(e_1, e_2 \ldots e_{n-1})$ that satisfy the ordinarity criterion $e_i = e_{n-i}$. Note that for such a tuple, the condition $\sum_{i=1}^{n-1} ie_i \equiv 0 \mod n$ is satisfied automatically. We now proceed to find an upper bound on this quantity, using a result of Bucur et. al. in \cite{BDFL09} that we will recall below. Let
	$$
    L_{n-2} = \prod_{j=1}^{n-2} \prod_Q \left(1 - \frac{j}{(\barsQ + 1)(\barsQ + j)} \right)
    $$
	 where the product is over monic irreducible polynomials $Q \in \F_q[x]$.\\
	
	\begin{theorem}[\cite{BDFL09}, Prop 4.3]
	Fix a tuple of positive integers $(e_1,e_2)$. Then, for any $\epsilon >0$ and as $q$ gets large,
	
	\begin{align*}
	\mid \! \mathcal{F}_{e_1, e_2} \! \mid = \frac{L_1 q^{e_1 + e_2}}{\zeta(2)^2} \left(1 + O(q^{-e_2(1-\epsilon)} + q^{-e_1/2}) \right)
	\end{align*}
    \end{theorem}
    
    \begin{remark}
        The number of monic polynomials of degree $d$ in $\F_q[x]$ is $q^d$ and the proportion of these that are squarefree is $(1 - 1/q)$. One might expect, similarly, that the proportion of pairs of monic polynomials of degrees $(e_1,e_2)$ that are squarefree and coprime, also form a positive proportion of the total number of pairs of monic polynomials, $q^{e_1+e_2}$. The above theorem shows that this is indeed the case. The next proposition shows that the same is true for $(e_1, e_2 \ldots e_{n-1})$ for any odd prime $n$, although with a weaker error term.
    \end{remark}

    For the next proposition, we refer the reader to \cite{BDFL09}, Corollary 7.2. 
    \begin{proposition}
    Fix a tuple of positive integers $(e_1, e_2 \ldots e_{n-1})$. Fix an $\epsilon >0$. Then, as $q$ gets large, 
    $$
    \mid \mathcal{F}_{(e_1, e_2 \ldots e_{n-1})} \mid  = \frac{L_{n-2} q^{e_1+e_2 \ldots e_{n-1}}}{\zeta(2)^{n-1}}(1 + O(q^{\epsilon(e_2 + \ldots e_{n-1}+q)+ (1-\epsilon)q}(q^{-e_2}+ \cdots q^{-e_{n-1}})  + q^{-(e_1 -3q)/2} ))
    $$
    \end{proposition}

    \begin{proof}
        Consider the expression given in \cite{BDFL09}, Corollary 7.2. Summing the expression over all possible partitions $m = k_1 + k_2 \ldots k_{n-1}$ gives:
         \begin{align*}
             &\frac{L_{n-2} q^{e_1 + e_2 \ldots e_{n-1}}}{\zeta(2)^{n-1}} \left(\frac{n-1}{q+n-1} \right)^m \left(\frac{q}{(q+n-1)(q-1)} \right)^{q-m}\\
             &\times(1 + O(q^{\epsilon(e_2 + \ldots e_{n-1}+q)+ (1-\epsilon)m}(q^{-e_2}+q^{-e_3} \cdots q^{-e_{n-1}})  + q^{-(e_1-m)/2 +q} ))
        \end{align*}
       
        Summing over all possibilities of $m$ now gives the result.
    \end{proof}

    Parsing these propositions tells us that for large enough $q$, 
    $$
    \mid \mathcal{F}_{(e_1, e_2 \ldots e_{n-1})} \mid \le K_1 q^{e_1 + e_2 \ldots e_{n-1}} + K_2 q^{e_1/2 + e_2 \ldots e_{n-1}} + \sum_{i=2}^{n-1} K_{3,i} q^{e_1 + \ldots \epsilon e_i + \ldots e_{n-1}}
    $$
    
    where $K_1, K_2$ and the $K_{3,i}$'s depend on $\epsilon$, $q$ and $n$, but are independent of the $e_j$'s. Since for $\epsilon <1$ the first term in the above expression is the largest, we let $K = \max(K_1, K_2 , K_{3,2} \ldots K_{3,n-1})$ and so for large enough $q$,
    $$
    \mid \mathcal{F}_{(e_1, e_2 \ldots e_{n-1})} \mid \le K q^{e_1+ e_2 \ldots e_{n-1}}.
    $$
    
    Thus \eqref{anumber0} implies that
    \begin{equation}
    \label{anum0bound}
    N^*(\Fs,0,X) \le K \left(\frac{q+n-1}{q} \right) \left( \sum_{q^{2(e_1+e_2 \ldots e_{(n-1)/2})}<X} q^{2(e_1+e_2 \ldots e_{(n-1)/2})}\right).
    \end{equation}
    
    The following lemma will be used to find an upper bound for the expression above. 
    
    \begin{lemma}
    \label{superellbounds}
    As $X$ gets large, 
        $$
        \sum_{q^{e_1+e_2 \ldots e_r}<X} q^{e_1+e_2 \ldots e_r} = O(X \log(X)^{r-1}).
        $$
    Here, the implied constants depend on $q$ and $r$.
    \end{lemma}
    
    \begin{proof}
        Consider the expression,
	$$
	\left(\frac{1}{1-qT} \right)^r.
	$$
	
	The coefficient of $T^m$ in this expression is $$ \sum_{e_1 + e_2 \ldots e_r = m} q^{e_1+ e_2 \ldots + e_r }.$$ On the other hand, by the binomial theorem, the corfficient of $T^m$ in $(1-qT)^{-r}$ is
%	\begin{displaymath}
%	    \sum_{k \ge 0} \begin{pmatrix} -r \\ k \end{pmatrix} (-qT)^{k}
%	\end{displaymath}
	 $$\begin{pmatrix} r+m-1 \\ r-1 \end{pmatrix}  q^m. $$ Further, 
	$$
	\begin{pmatrix} r+m-1 \\ r-1 \end{pmatrix} \le \frac{(m+r)^{r-1}}{(r-1)!}.
	$$
	
	Therefore, we have that,
    	\begin{align*}
    	    \sum_{q^{e_1+e_2 \ldots e_{r}}<X} q^{e_1+e_2 \ldots e_r} &= \sum_{q^m < X} \sum_{e_1 + e_2 \ldots e_r = m} q^{e_1 + e_2 \ldots e_r} \\
    	    &= \sum_{q^m < X} \begin{pmatrix} r+m-1 \\ r-1 \end{pmatrix}  q^m\\
	        %&= \sum_{m<r \atop q^m < X} \frac{(m+r)^{r-1}}{(r-1)!} q^m + \sum_{m \ge r \atop q^m < X} \frac{(m+r)^{r-1}}{(r-1)!} q^m\\
	        &\le \frac{(2r)^{r-1}}{(r-1)!} \sum_{m<r \atop q^m < X} q^m + \sum_{m \ge r \atop q^m < X} \frac{(2m)^{r-1}}{(r-1)!} q^m\\
	        &= D_r X \log(X)^{r-1} + O(X\log(X)^{r-2})
	    \end{align*}
    
    where the last step follows by Euler Summation. This proves the lemma.
    \end{proof}

    \begin{proposition}
    \label{superellprop0}
    For large enough $q$ (depending only on $n$),
        $$
        N^*(\Fs,0,X) = O(X \log(X)^{\frac{n-3}{2}})
        $$
        Hence, $N(\Fs,0,X) = O(X^{2/(n-1)} \log(X)^{\frac{n-3}{2}})$, where the implied constants depend on $q$ and $n$.
    \end{proposition}
    
    \begin{proof}
        To obtain the first statement, we use Equation \eqref{anum0bound}:
        $$
      N^*(\Fs,0,X) \le K \left(\frac{q+n-1}{q} \right) \left( \sum_{q^{2(e_1+e_2 \ldots e_{(n-1)/2})}<X} q^{2(e_1+e_2 \ldots e_{(n-1)/2})} \right)
        $$
        
        and Lemma \ref{superellbounds}, with $q$ replaced by $q^2$. The second part of the statement follows from the fact that $N(\Fs,0,X) = N^*(\Fs,0,q^2X^{2/(n-1)})$.
        
    \end{proof}

	 We remind the reader here that for the quantity that we are interested in, namely the probability that a superelliptic curve is ordinary, $q$ and $n$ are fixed. Therefore, the fact that the implied constants above depend on $q$ and $n$ will make no difference to the theorem below.

	\begin{theorem}
	\label{Superellmaintheorem}
	    The probability that a superelliptic curve $y^n = f(x)$ over $\F_q$ with $n$ prime and $q$ large enough, is ordinary, is zero. That is,
	    $$
	    \lim_{X \rightarrow \infty} \frac{N(\Fs,0,X)}{N(\Fs,X)} = 0.
	    $$
	\end{theorem}
	
	\begin{proof}
	    By Proposition \ref{superellprop0}, the numerator, $N(\Fs,0,X)$ is bounded above by $X^{2/(n-1)} \log(X)^{\frac{n-3}{2}}$. By Corollary \ref{superellpropmain}, the denominator is bounded below by $X^{2/(n-1)} \log(X)^{n-2}$. This proves the theorem.
	    
	\end{proof}
	
	\begin{remark}
	\label{EndremarkSE}
	      It is interesting to note that for a given $g$, the space of superelliptic curves of degree $n$ and genus $g$ decomposes over $\overline{\F}_q$ into irreducible components that correspond to partitions of $\mathbf{m} = \sum_{i=1}^{n-1} e_i$ such that $\sum_{i=1}^{n-1}{ie_i} \equiv 0 \mod n$. The ordinary locus intersects a small proportion of these components. For $n=3$, for instance, it only intersects one component. A similar thing was true for the Artin-Schreier locus $\mathscr{AS}_g$. For fixed $p$-rank $s$, one can obtain a combinatorial description of the components contained in the stratum $\mathscr{AS}_{g,s}$ (\cite{PriesZhu12}). One can ask if a similar result holds for the $a$-number strata of the moduli space of superelliptic curves in characteristic 2.
	\end{remark}

\section{Numerical Data}
\label{section:data}
\subsection{Artin-Schreier curves in characteristic $p$}

Here we list some values of constants calculated in \S 3.  Recall that 
$$
P(\mathscr{AS},0) := \left(\frac{1-q^{-1}+q^{-2}}{1+q^{-1}} \right) \phi(1) \zeta(2)% = 1 - 3q^{-1} + 6q^{-2} + O(q^{-3})
$$
is the probability that an Artin-Schreier curve in characteristic 2 as in \S 3 is ordinary (Corollary \ref{ArtinSchreierMainTheorem}). For brevity, we let $\varphi(q) = \prod_{i=1}^{\infty}(1 + q^{-i})^{-1}$ the constant predicted in \cite{CEZB13}.\\

\begin{center}
\begin{tabular}{|c|c|c|c|c|}
    \hline
    $p$ & $q$ & $\phi(1)$  & $P(\mathscr{AS})$ & $\varphi(q)$ \\
    \hline
     2& 2& 0.314148  &  0.314148 &0.419422 \\
     \hline
     2& 4& 0.593976 & 0.514777 & 0.737512\\
     \hline
     2& 8& 0.776577& 0.702617 & 0.873264\\
     \hline
     2& 16& 0.882162  &  0.833730 & 0.937270\\
     \hline
     2& 32& 0.939367 & 0.911820 & 0.968720\\
     \hline
\end{tabular}
\captionof{table}{Constants for Artin-Schreier curves} \label{Table1}
\end{center}

\bibliographystyle{plain}
\bibliography{References}
\end{document}